\documentclass[11pt]{article}
\usepackage{amsmath}
\usepackage{amssymb}
\usepackage{amsthm}
\usepackage{bookmark}
\usepackage{geometry}
\usepackage{hyperref}
\newtheorem{theorem}{Theorem}
\newtheorem{corollary}[theorem]{Corollary}
\newtheorem{lemma}[theorem]{Lemma}
\newtheorem{proposition}[theorem]{Proposition}
\theoremstyle{definition}
\newtheorem{example}[theorem]{Example}
\newtheorem{remark}[theorem]{Remark}
\DeclareMathOperator{\im}{im}
\DeclareMathOperator{\lcm}{lcm}
\DeclareMathOperator{\rank}{rank}
\renewcommand{\labelenumi}{\upshape(\roman{enumi})}
\renewcommand{\theenumi}\labelenumi

\title{Classification of quadratic forms over finite fields with maximal and minimal Artin-Schreier curves}
\author{Ruikai Chen}
\date{\small School of Mathematical Sciences, South China Normal University, Guangzhou 510631, China\\Email: \href{mailto:chen.rk@outlook.com}{chen.rk@outlook.com}}

\begin{document}

\maketitle

\noindent\textbf{Abstract.} This paper explores quadratic forms over finite fields with associated Artin-Schreier curves. Specifically, we investigate quadratic forms of $\mathbb F_{q^n}/\mathbb F_q$ represented by polynomials over $\mathbb F_{q^n}$ with $q$ odd, characterizing them using certain matrices defined by coefficients of the polynomials. In particular, a comprehensive treatment will be given for those polynomials whose coefficients all lie in $\mathbb F_q$. Afterwards, the results on quadratic forms will be applied to get maximal and minimal Artin-Schreier curves explicitly.

\section{Introduction}

Given $\mathbb F_{q^n}/\mathbb F_q$, an extension of a finite field of odd characteristic, and the polynomial ring $\mathbb F_{q^n}[x]$, define
\[\mathrm{Tr}(x)=\sum_{i=0}^{n-1}x^{q^i}.\]
Let $\ell$ be a nonzero $q$-linear polynomial over $\mathbb F_{q^n}$ of the form
\[\ell(x)=\sum_{i=0}^{n-1}a_ix^{q^i}.\]
Here a $q$-linear polynomial over $\mathbb F_{q^n}$ is also viewed as a linear endomorphism of $\mathbb F_{q^n}/\mathbb F_q$. Then $\mathrm{Tr}(x\ell(x))$ as a map from $\mathbb F_{q^n}$ to $\mathbb F_q$ is a quadratic form of $\mathbb F_{q^n}/\mathbb F_q$ (in fact, every quadratic form of $\mathbb F_{q^n}/\mathbb F_q$ can be written in this form), and it induces a homogeneous polynomial over $\mathbb F_q$ of degree $2$ in $n$ indeterminates or a zero polynomial. Quadratic forms over finite fields have attracted significant attention in number theory, with applications in various areas including coding theory, cryptography and combinatorics. Furthermore, they are directly related to the number of rational points of some Artin-Schreier curves over finite fields. More precisely, consider the projective curve over $\mathbb F_{q^n}$ defined by
\[y^q-y-x\ell(x)\in\mathbb F_{q^n}[x,y].\]
If the degree of $\ell$ is $q^m$, then the genus of this curve is $\frac{(q-1)q^m}2$ (see \cite[Proposition 3.7.10]{stichtenoth2009}). For the number $N$ of rational points of the above curve, the well-known Hasse-Weil bound asserts that
\[q^n+1-(q-1)q^{\frac n2+m}\le N\le q^n+1+(q-1)q^{\frac n2+m}.\]
The curve is called a maximal curve if $N$ attains the upper bound, and called a minimal curve if $N$ attains the lower bound. We have special interests in such curves for their unique properties.

There has been a lot of research concerning quadratic forms over finite fields and related maximal or minimal Artin-Schreier curves. To name a few, in \cite{klapper1997cross} and \cite{fitzgerald2017invariants}, the quadratic forms represented by monomials are classified. Also, one may refer to \cite{coulter2002number} for those curves defined by monomials, and to \cite{anbar2014quadratic} for those curves defined by polynomials with coefficients in prime fields. Some other results can be found in \cite{ccakccak2008curves,ozbudak2011quadratic,ozbudak2012quadratic,ozbudak2016explicit,bartoli2021explicit,oliveira2023number,oliveira2023artin}.

In this paper, we will present a general approach in the language of quadratic forms to determine the number of rational points of those Artin-Schreier curves. Given a basis $\alpha_1,\dots,\alpha_n$ of $\mathbb F_{q^n}/\mathbb F_q$, for $w=w_1\alpha_1+\dots+w_n\alpha_n$ with $w_1,\dots,w_n\in\mathbb F_q$ we have
\[\mathrm{Tr}(w\ell(w))=\begin{pmatrix}w_1&\cdots&w_n\end{pmatrix}C\begin{pmatrix}w_1\\\vdots\\w_n\end{pmatrix},\]
where $C$ is an $n\times n$ matrix whose $(i,j)$ entry is $\frac12\mathrm{Tr}(\alpha_i\ell(\alpha_j)+\alpha_j\ell(\alpha_i))$. We call $C$ the associated matrix of the quadratic form $\mathrm{Tr}(x\ell(x))$ with respect to the basis $\alpha_1,\dots,\alpha_n$. Since $C$ is symmetric, it can be diagonalized as $C=P^\mathrm TDP$ for some nonsingular matrix $P$ and some diagonal matrix $D$ over $\mathbb F_q$ (cf. \cite[Theorem 6.21]{lidl1997}). Let $r=\rank C$ be the rank of $C$, and rearrange the entries of $D$ so that its first $r$ diagonal entries $d_1,\dots,d_r$ are nonzero. Then
\begin{equation}\label{eq13}\mathrm{Tr}(w\ell(w))=\begin{pmatrix}w_1&\cdots&w_n\end{pmatrix}P^\mathrm TDP\begin{pmatrix}w_1\\\vdots\\w_n\end{pmatrix}=d_1w_1^2+\dots+d_rw_r^2.\end{equation}
Let $\chi$ be the canonical additive character of $\mathbb F_q$, and $\eta$ be the quadratic character of $\mathbb F_q^*$. For $a\in\mathbb F_q^*$, we have
\[\sum_{u\in\mathbb F_q}\chi(au^2)=\eta(a)\mathcal G,\]
where
\[\mathcal G=\sum_{u\in\mathbb F_q^*}\chi(u)\eta(u)\]
is the quadratic Gauss sum of $\mathbb F_q$ with $\mathcal G^2=\eta(-1)q$. Replacing $\begin{pmatrix}w_1&\cdots&w_n\end{pmatrix}P^\mathrm T$ by $\begin{pmatrix}w_1&\cdots&w_n\end{pmatrix}$ in \eqref{eq13}, we get
\[\begin{split}&\mathrel{\phantom{=}}\sum_{w\in\mathbb F_{q^n}}\chi(\mathrm{Tr}(w\ell(w)))\\&=\sum_{w_1,\dots,w_n\in\mathbb F_q}\chi(d_1w_1^2+\dots+d_rw_r^2)\\&=\sum_{w_1\in\mathbb F_q}\chi(d_1w_1^2)\cdots\sum_{w_r\in\mathbb F_q}\chi(d_rw_r^2)\sum_{w_{r+1},\dots,w_n\in\mathbb F_q}1\\&=\eta(d_1\cdots d_r)\mathcal G^rq^{n-r}.\end{split}\]
Here the numbers $r$, called the rank of the quadratic form, and $\eta(d_1\cdots d_r)$, denoted by $\epsilon$, act as invariants of the quadratic form $\mathrm{Tr}(x\ell(x))$, which do not depend on the choice of basis. In particular, if $C$ is nonsingular, then
\[\epsilon=\eta(d_1\cdots d_r)=\eta(\det D)=\eta(\det C)\]
since $C=P^\mathrm TDP$ with $\det P\in\mathbb F_q^*$.

Now we are ready for the number of roots of $\mathrm{Tr}(x\ell(x))+\lambda$ in $\mathbb F_{q^n}$ for $\lambda\in\mathbb F_q$, which is, by the orthogonality relations for characters,
\[\begin{split}&\mathrel{\phantom{=}}q^{-1}\sum_{w\in\mathbb F_{q^n}}\sum_{u\in\mathbb F_q}\chi(u\mathrm{Tr}(w\ell(w))+u\lambda)\\&=q^{-1}\sum_{u\in\mathbb F_q}\sum_{w_1,\dots,w_n\in\mathbb F_q}\chi(u(d_1w_1^2+\dots+d_rw_r^2))\chi(u\lambda)\\&=q^{n-1}+q^{-1}\sum_{u\in\mathbb F_q^*}\eta(u)^r\epsilon\mathcal G^rq^{n-r}\chi(u\lambda);\end{split}\]
that is
\[\begin{cases}q^{n-1}+(-1)^\frac{(q-1)r}4\epsilon(q-1)q^{n-1-\frac r2}&\text{if }r\text{ is even},\\q^{n-1}&\text{if }r\text{ is odd}\end{cases}\]
when $\lambda=0$, and is
\[\begin{cases}q^{n-1}-(-1)^\frac{(q-1)r}4\epsilon q^{n-1-\frac r2}&\text{if }r\text{ is even},\\q^{n-1}+(-1)^\frac{(q-1)(r+1)}4\eta(\lambda)\epsilon q^{n-1-\frac{r-1}2}&\text{if }r\text{ is odd}\end{cases}\]
when $\lambda\ne0$. For the Artin-Schreier curve over $\mathbb F_{q^n}$ defined by $y^q-y-x\ell(x)$, it has exactly one point at infinity. Given $x_0\in\mathbb F_{q^n}$, $y_0^q-y_0=x_0\ell(x_0)$ for some $y_0\in\mathbb F_{q^n}$ if and only if $\mathrm{Tr}(x_0\ell(x_0))=0$, in which case the number of such $y_0$ in $\mathbb F_{q^n}$ is $q$. Therefore, we obtain the following conclusion (cf. \cite[Corollary 3.7]{anbar2014quadratic}).% $\mathrm{Tr}(x\ell(x))$ maps $\mathbb F_{q^n}$ onto $\mathbb F_q$ when $r>1$. The invariants $r$ and $\epsilon$ play an important role. Two quadratic forms of $\mathbb F_{q^n}/\mathbb F_q$ are called equivalent if they have the same matrix by change of basis. By \cite[Lemma 6.20]{lidl1997}, one can show by induction that $\mathrm{Tr}(x\ell(x))$ is equivalent to a quadratic form with associated matrix either
% \[\begin{pmatrix}1\\&\ddots\\&&1\\&&&1\end{pmatrix}\quad\text{or}\quad\begin{pmatrix}1\\&\ddots\\&&1\\&&&a\end{pmatrix}\]
% for some non-square element $a$ in $\mathbb F_{q^n}$.

\begin{lemma}\label{maximal}
The Artin-Schreier curve over $\mathbb F_{q^n}$ defined by $y^q-y-x\ell(x)$ with $\deg\ell=q^m$ is maximal or minimal if and only if $n$ is even and $r=n-2m$. In this case, it is maximal if $\epsilon=(-1)^\frac{(q-1)r}4$, and is minimal otherwise.
\end{lemma}

We have seen that the invariants $r$ and $\epsilon$ carry essential information about the quadratic form, which comprise the central topic of this paper. In Section \ref{class}, we propose a general method to determine the invariants so as to classify all quadratic forms represented by polynomials over finite fields. In Section \ref{base}, we focus on a special case where all the coefficients of the polynomials belong to $\mathbb F_q$, providing a complete characterization. In Section \ref{curve}, we study some polynomials of simple forms inducing maximal or minimal Artin-Schreier curves.

The following notation will be used throughout. For $\mathbb F_{q^n}/\mathbb F_q$ with $n>1$ and $q$ odd, let $\ell$ be a nonzero $q$-linear polynomial over $\mathbb F_{q^n}$ of degree less than $q^n$. Denote by $\mathrm{Tr}$ and $\mathrm N$ the trace map and the norm map of $\mathbb F_{q^n}/\mathbb F_q$ respectively. The invariant $\epsilon$ of the quadratic form $\mathrm{Tr}(x\ell(x))$ is defined above, and $\eta$ is the quadratic character of $\mathbb F_q^*$.

Consider $\mathbb F_{q^n}$ as a vector space over $\mathbb F_q$. For a subspace $W$ of $\mathbb F_{q^n}$, denote by $W^\perp$ the subspace consisting of all elements $\alpha\in\mathbb F_{q^n}$ such that $\mathrm{Tr}(\alpha x)$ vanishes on $W$ (one can check that $(W^\perp)^\perp=W$ and $\dim_{\mathbb F_q}W+\dim_{\mathbb F_q}W^\perp=n$). Given a $q$-linear polynomial $L$ over $\mathbb F_{q^n}$ of the form
\[L(x)=\sum_{i=0}^{n-1}a_ix^{q^i},\]
define $L^\prime$ as
\[L^\prime(x)=a_0x+\sum_{i=1}^{n-1}(a_ix)^{q^{n-i}},\]
and let
\[M_L=\begin{pmatrix}a_0&a_1&\cdots&a_{n-1}\\a_{n-1}^q&a_0^q&\cdots&a_{n-2}^q\\\vdots&\vdots&\ddots&\vdots\\a_1^{q^{n-1}}&a_2^{q^{n-1}}&\cdots&a_0^{q^{n-1}}\end{pmatrix}.\]
For a positive integer $k$ and a matrix $M$ over some field, let $M^{(k)}$ denote the submatrix of $M$ formed by its first $k$ rows and $k$ columns.

\section{Classification of quadratic forms by matrices}\label{class}

For $L$ and $L^\prime$ defined above, it is clear that $\mathrm{Tr}(\alpha L(\beta))=\mathrm{Tr}(L^\prime(\alpha)\beta)$ for all $\alpha,\beta\in\mathbb F_{q^n}$. In addition, $L^\prime(\alpha)=0$ if and only if $\mathrm{Tr}(\alpha L(\beta))=0$ for all $\beta\in\mathbb F_{q^n}$, if and only if $\alpha\in(\im L)^\perp$. This means $\ker L^\prime=(\im L)^\perp$.

Let $\alpha_1,\dots,\alpha_n$ be a basis of $\mathbb F_{q^n}/\mathbb F_q$ with a matrix
\begin{equation}\label{A}A=\begin{pmatrix}\alpha_1&\alpha_2&\cdots&\alpha_n\\\alpha_1^q&\alpha_2^q&\cdots&\alpha_n^q\\\vdots&\vdots&\ddots&\vdots\\\alpha_1^{q^{n-1}}&\alpha_2^{q^{n-1}}&\cdots&\alpha_n^{q^{n-1}}\end{pmatrix}.\end{equation}
Then the $(i,j)$ entry of $A^\mathrm TM_LA$ is
\[\begin{split}\begin{pmatrix}\alpha_i&\alpha_i^q&\cdots&\alpha_i^{q^{n-1}}\end{pmatrix}M_L\begin{pmatrix}\alpha_j\\\alpha_j^q\\\vdots\\\alpha_j^{q^{n-1}}\end{pmatrix}&=\begin{pmatrix}\alpha_i&\alpha_i^q&\cdots&\alpha_i^{q^{n-1}}\end{pmatrix}\begin{pmatrix}L(\alpha_j)\\L(\alpha_j)^q\\\vdots\\L(\alpha_j)^{q^{n-1}}\end{pmatrix}\\&=\mathrm{Tr}(\alpha_iL(\alpha_j)).\end{split}\]
For the quadratic form $\mathrm{Tr}(x\ell(x))$, its associated symmetric matrix $C$ with respect to the basis $\alpha_1,\dots,\alpha_n$ has $(i,j)$ entry
\[\frac12\mathrm{Tr}(\alpha_i\ell(\alpha_j)+\alpha_j\ell(\alpha_i))=\frac12\mathrm{Tr}(\alpha_i\ell(\alpha_j)+\alpha_i\ell^\prime(\alpha_j)).\]
If $L=\frac{\ell+\ell^\prime}2$, then clearly $C=A^\mathrm TM_LA$, and the rank of the quadratic form is $\rank C=\rank M_L$ since $A$ is nonsingular. This will be employed to characterize the quadratic form in terms of the coefficients of $\ell$. For $\rank M_L$, we have the following results.

\begin{theorem}\label{rank}
Let
\[L(x)=\sum_{i=0}^{n-1}a_ix^{q^i}\]
with coefficients in $\mathbb F_{q^n}$. For an integer $k$ with $0<k\le n$, the following conditions are equivalent:
\begin{enumerate}
\item\label{rank_1}$\dim_{\mathbb F_q}\ker L\ge k$;
\item\label{rank_2}$\rank M_L\le n-k$;
\item\label{rank_3}the first $n-k+1$ rows of $M_L$ are linearly dependent over $\mathbb F_{q^n}$.
\end{enumerate}
As a consequence, $\dim_{\mathbb F_q}\ker L+\rank M_L=n$, and the first $\rank M_L$ number of rows of $M_L$ are linearly independent over $\mathbb F_{q^n}$ while the first $\rank M_L+1$ number of rows are not.
\end{theorem}
\begin{proof}
First, suppose that $\dim_{\mathbb F_q}\ker L\ge k$. Take a basis of $\ker L/\mathbb F_q$ and extend it to a basis of $\mathbb F_{q^n}/\mathbb F_q$, denoted by $\alpha_1,\dots,\alpha_n$, so that
\[M_L\begin{pmatrix}\alpha_1&\alpha_2&\cdots&\alpha_n\\\alpha_1^q&\alpha_2^q&\cdots&\alpha_n^q\\\vdots&\vdots&\ddots&\vdots\\\alpha_1^{q^{n-1}}&\alpha_2^{q^{n-1}}&\cdots&\alpha_n^{q^{n-1}}\end{pmatrix}=\begin{pmatrix}L(\alpha_1)&L(\alpha_2)&\cdots&L(\alpha_n)\\L(\alpha_1)^q&L(\alpha_2)^q&\cdots&L(\alpha_n)^q\\\vdots&\vdots&\ddots&\vdots\\L(\alpha_1)^{q^{n-1}}&L(\alpha_2)^{q^{n-1}}&\cdots&L(\alpha_n)^{q^{n-1}}\end{pmatrix},\]
where the matrix on the right side has at least $k$ columns being zero. Then it has rank at most $n-k$, and so does $M_L$.

Next, \ref{rank_2} implies \ref{rank_3} since the rank of $M_L$ is the maximal number of its rows that are linearly independent over $\mathbb F_{q^n}$.

Suppose now that the first $n-k+1$ rows of $M_L$ are linearly dependent over $\mathbb F_{q^n}$. Then
\[\begin{pmatrix}b_0&b_1&\cdots&b_{n-k}&0&\cdots&0\end{pmatrix}M_L\]
is zero for some $b_0,b_1,\dots,b_{n-k}\in\mathbb F_{q^n}$. Note that
\begin{equation}\label{eq7}\begin{split}&\mathrel{\phantom{=}}\begin{pmatrix}b_0&b_1&\cdots&b_{n-k}&0&\cdots&0\end{pmatrix}M_L\\&=\begin{pmatrix}b_0&b_1&\cdots&b_{n-k}&0&\cdots&0\end{pmatrix}\begin{pmatrix}a_0&a_1&\cdots&a_{n-1}\\a_{n-1}^q&a_0^q&\cdots&a_{n-2}^q\\\vdots&\vdots&\ddots&\vdots\\a_1^{q^{n-1}}&a_2^{q^{n-1}}&\cdots&a_0^{q^{n-1}}\end{pmatrix}\\&=\begin{pmatrix}\sum_{i=0}^{n-k}b_ia_{-i}^{q^i}&\sum_{i=0}^{n-k}b_ia_{1-i}^{q^i}&\cdots&\sum_{i=0}^{n-k}b_ia_{n-1-i}^{q^i}\end{pmatrix},\end{split}\end{equation}
where the subscripts are taken modulo $n$. Therefore, for $L_0(x)=\sum_{i=0}^{n-k}b_ix^{q^i}$, one has
\begin{equation}\label{eq8}L_0(L(x))=\sum_{i=0}^{n-k}b_i\sum_{j=0}^{n-1}a_j^{q^i}x^{q^{i+j}}\equiv\sum_{j=0}^{n-1}\sum_{i=0}^{n-k}b_ia_{j-i}^{q^i}x^{q^j}\equiv0\pmod{x^{q^n}-x},\end{equation}
and then $\im L\subseteq\ker L_0$. As
\[\dim_{\mathbb F_q}\im L\le\dim_{\mathbb F_q}\ker L_0\le n-k,\]
it turns out that $\dim_{\mathbb F_q}\ker L\ge k$. This indicates that \ref{rank_3} implies \ref{rank_1}.
\end{proof}

For $L=\frac{\ell+\ell^\prime}2$, we will see that $\epsilon$ depends on some submatrix of $M_L$. In fact, for $r=\dim_{\mathbb F_q}\im L$, the submatrix $M_L^{(r)}$ is nonsingular and its determinant will be used to determine $\epsilon$.

\begin{theorem}\label{det}
Let $L=\frac{\ell+\ell^\prime}2$, and suppose that $\dim_{\mathbb F_q}\im L=r$ for some integer $r$ with $0<r\le n$. If $r=n$, then $\det M_L\in\mathbb F_q$ and $\epsilon=(-1)^{n-1}\eta(\det M_L)$. Suppose that $r<n$. Then $\det M_L^{(r)}\ne0$, and
\begin{enumerate}
\item\label{det_1}$\Big(\det M_L^{(r)}\Big)^{q+1}=\Big(\det M_L^{(r,1)}\Big)^2$, where $M_L^{(r,1)}$ is the submatrix formed by the $r$ consecutive rows starting from the second row of $M_L$ along with the first $r$ columns of $M_L$, and $\epsilon=1$ if and only if $\Big(\det M_L^{(r)}\Big)^\frac{q+1}2=\det M_L^{(r,1)}$; 
\item\label{det_2}$\epsilon=1$ if and only if $\det M_L^{(r)}$ is a square in $\mathbb F_{q^n}$, when $n$ is odd;
\item\label{det_3}$a^{2q^r\frac{q^m-1}{q-1}}\det M_L^{(r)}\in\mathbb F_q$ and $\epsilon=(-1)^{n-1}\eta\Big(a^{2q^r\frac{q^m-1}{q-1}}\det M_L^{(r)}\Big)$, when $r=n-2m$ and $\ell$ has degree $q^m$ with leading coefficient $a$.
\end{enumerate}
\end{theorem}
\begin{proof}
Take a basis $\alpha_1,\dots,\alpha_n$ of $\mathbb F_{q^n}/\mathbb F_q$ such that $\alpha_{r+1},\dots,\alpha_n\in\ker L$, and define $A$ as in \eqref{A}. By the discussion preceding Theorem \ref{rank}, the quadratic form $\mathrm{Tr}(x\ell(x))$ has an associated symmetric matrix written as
\[\begin{pmatrix}C&\mathbf0\\\mathbf0&\mathbf0\end{pmatrix}=A^\mathrm TM_LA,\]
where $C$ is an $r\times r$ submatrix over $\mathbb F_q$ and other entries are all zero. Let $C=P^\mathrm TDP$ for some nonsingular matrix $P$ and some diagonal matrix $D$ over $\mathbb F_q$, so that
\[\begin{pmatrix}C&\mathbf0\\\mathbf0&\mathbf0\end{pmatrix}=\begin{pmatrix}P^\mathrm T&\mathbf0\\\mathbf0&I_{n-r}\end{pmatrix}\begin{pmatrix}D&\mathbf0\\\mathbf0&\mathbf0\end{pmatrix}\begin{pmatrix}P&\mathbf0\\\mathbf0&I_{n-r}\end{pmatrix},\]
where $I_{n-r}$ is the identity matrix of order $n-r$ over $\mathbb F_q$. It follows from Theorem \ref{rank} that $\rank C=\rank M_L=r$, so the diagonal entries of $D$ are all nonzero and $\epsilon=\eta(\det D)=\eta(\det C)$.

Observe that
\[(\det A)^q=\begin{vmatrix}\alpha_1^q&\alpha_2^q&\cdots&\alpha_n^q\\\alpha_1^{q^2}&\alpha_2^{q^2}&\cdots&\alpha_n^{q^2}\\\vdots&\vdots&\ddots&\vdots\\\alpha_1&\alpha_2&\cdots&\alpha_n\end{vmatrix}=(-1)^{n-1}\det A,\]
so $(\det A)^2\in\mathbb F_q$. If $r=n$, then $C=A^\mathrm TM_LA$, and
\[\epsilon=\eta((\det A)^2\det M_L)=(-1)^{n-1}\eta(\det M_L).\]

Assume now that $r<n$ and let $\beta_1,\dots,\beta_n$ be the dual basis of $\alpha_1,\dots,\alpha_n$ in $\mathbb F_{q^n}/\mathbb F_q$, so that
\[A^{-1}=\begin{pmatrix}\beta_1&\beta_1^q&\cdots&\beta_1^{q^{n-1}}\\\beta_2&\beta_2^q&\cdots&\beta_2^{q^{n-1}}\\\vdots&\vdots&\ddots&\vdots\\\beta_n&\beta_n^q&\cdots&\beta_n^{q^{n-1}}\end{pmatrix}.\]
Partition it as
\[A^{-1}=\begin{pmatrix}B_{11}&B_{12}\\B_{21}&B_{22}\end{pmatrix},\]
where $B_{11}$ is an $r\times r$ submatrix. Then
\[\begin{split}M_L&=\big(A^\mathrm T\big)^{-1}\begin{pmatrix}C&\mathbf0\\\mathbf0&\mathbf0\end{pmatrix}A^{-1}\\&=\begin{pmatrix}B_{11}^\mathrm T&B_{21}^\mathrm T\\B_{12}^\mathrm T&B_{22}^\mathrm T\end{pmatrix}\begin{pmatrix}C&\mathbf0\\\mathbf0&\mathbf0\end{pmatrix}\begin{pmatrix}B_{11}&B_{12}\\B_{21}&B_{22}\end{pmatrix}\\&=\begin{pmatrix}B_{11}^\mathrm TC&\mathbf0\\B_{12}^\mathrm TC&\mathbf0\end{pmatrix}\begin{pmatrix}B_{11}&B_{12}\\B_{21}&B_{22}\end{pmatrix}\\&=\begin{pmatrix}B_{11}^\mathrm TCB_{11}&B_{11}^\mathrm TCB_{12}\\B_{12}^\mathrm TCB_{11}&B_{12}^\mathrm TCB_{12}\end{pmatrix},\end{split}\]
and
\begin{equation}\label{eq5}\det M_L^{(r)}=(\det B_{11})^2\det C.\end{equation}
Note that
\[\begin{split}&\mathrel{\phantom{=}}\begin{vmatrix}\beta_1&\beta_1^q&\cdots&\beta_1^{q^{r-1}}&\beta_1^{q^r}\\\beta_2&\beta_2^q&\cdots&\beta_2^{q^{r-1}}&\beta_2^{q^r}\\\vdots&\vdots&\ddots&\vdots&\vdots\\\beta_r&\beta_r^q&\cdots&\beta_r^{q^{r-1}}&\beta_r^{q^r}\\x&x^q&\cdots&x^{q^{r-1}}&x^{q^r}\end{vmatrix}\\&=\det B_{11}x^{q^r}+\dots+(-1)^r(\det B_{11})^qx,\end{split}\]
as a linear endomorphism of $\mathbb F_{q^n}/\mathbb F_q$ whose kernel is exactly the subspace spanned by $\beta_1,\dots,\beta_r$, as is easily verified. Also, $\det B_{11}$ is nonzero, and so is $\det M_L^{(r)}$. Divide the above polynomial by $\det B_{11}$ to get a monic polynomial
\[L_0(x)=x^{q^r}+\dots+(-1)^r(\det B_{11})^{q-1}x.\]
The statement \ref{det_2} follows immediately from \eqref{eq5} since $(\det C)^\frac{q^n-1}2=(\det C)^\frac{q-1}2$ when $n$ is odd. Since $L=L^\prime$, we have $\im L=(\ker L)^\perp=\ker L_0$, and
\[L_0(L(x))\equiv0\pmod{x^{q^n}-x}.\]
By the same argument as in \eqref{eq7} and \eqref{eq8}, we find that
\[\begin{pmatrix}(-1)^r(\det B_{11})^{q-1}&\cdots&1\end{pmatrix}\begin{pmatrix}\\&M_L^{(r)}&\\\\&\upsilon\end{pmatrix}\]
is zero, where $\upsilon$ is the $1\times r$ submatrix formed by the first $r$ entries of the $(r+1)$-st row of $M_L$. Cramer's rule leads to
\[(-1)^r(\det B_{11})^{q-1}=(-1)^r\frac{\det M_L^{(r,1)}}{\det M_L^{(r)}}.\]
This, coupled with
\[\Big(\det M_L^{(r)}\Big)^\frac{q-1}2=(\det B_{11})^{q-1}(\det C)^\frac{q-1}2,\]
implies \ref{det_1}.

To prove \ref{det_3}, suppose that $\dim_{\mathbb F_q}\ker L=2m$ and
\[\ell(x)+\ell^\prime(x)=\cdots+(ax)^{q^{n-m}}+ax^{q^m}+\cdots.\]
Let $L_1$ be the polynomial over $\mathbb F_{q^n}$ of degree less than $q^n$ such that
\[L_1^\prime(x)\equiv2a^{-q^{n-2m}}L(x)^{q^{n-m}}\pmod{x^{q^n}-x};\]
that is
\[L_1^\prime(x)=\cdots+x^{q^{n-2m}}+a^{q^{n-m}-q^{n-2m}}x,\]
with
\[L_1(x)=x^{q^{2m}}+\dots+a^{q^{n-m}-q^{n-2m}}x.\]
Then $\ker L_1^\prime=\ker L$, spanned by $\alpha_{r+1},\dots,\alpha_n$. Accordingly, $\im L_1=(\ker L_1^\prime)^\perp=\ker L_0$ and
\[L_0(L_1(x))\equiv0\pmod{x^{q^n}-x}.\]
It has been seen that $L_0$ is monic of degree $q^r$, and $L_1$ is monic of degree $q^{2m}$ by a simple investigation, which means
\[L_0(L_1(x))=x^{q^n}-x.\]
Equating the constant terms gives
\[(-1)^r(\det B_{11})^{q-1}a^{q^r(q^m-1)}=-1.\]
Finally, taking \eqref{eq5} into account we obtain
\[a^{2q^r\frac{q^m-1}{q-1}}\det M_L^{(r)}=a^{2q^r\frac{q^m-1}{q-1}}(\det B_{11})^2\det C\in\mathbb F_q,\]
where
\[\Big(a^{2q^r\frac{q^m-1}{q-1}}(\det B_{11})^2\Big)^\frac{q-1}2=a^{q^r(q^m-1)}(\det B_{11})^{q-1}=(-1)^{r-1},\]
so
\[\eta\Big(a^{2q^r\frac{q^m-1}{q-1}}\det M_L^{(r)}\Big)=(-1)^{r-1}\eta(\det C)=(-1)^{n-1}\epsilon,\]
as desired.
\end{proof}

\begin{example}
We classify all those quadratic forms of $\mathbb F_{q^4}/\mathbb F_q$ as follows (the results for $\mathbb F_{q^2}/\mathbb F_q$ and $\mathbb F_{q^3}/\mathbb F_q$ can be found in \cite[Section 3]{chen2023evaluation}). Let $\ell(x)=2ax^{q^2}+2bx^q+cx$ and $L(x)=\frac{\ell(x)+\ell^\prime(x)}2=(bx)^{q^3}+\big(a^{q^2}+a\big)x^{q^2}+bx^q+cx$ be nonzero polynomials over $\mathbb F_{q^4}$, so that
\[M_L=\begin{pmatrix}c&b&a^{q^2}+a&b^{q^3}\\b&c^q&b^q&a^{q^3}+a^q\\a^{q^2}+a&b^q&c^{q^2}&b^{q^2}\\b^{q^3}&a^{q^3}+a^q&b^{q^2}&c^{q^3}\end{pmatrix}.\]
Then
\begin{itemize}
\item $\rank M_L=1$ if and only if
\[\begin{vmatrix}c&b\\b&c^q\end{vmatrix}=\begin{vmatrix}b&a^{q^2}+a\\c^q&b^q\end{vmatrix}=\begin{vmatrix}a^{q^2}+a&b^{q^3}\\b^q&a^{q^3}+a^q\end{vmatrix}=0,\]
in which case $\epsilon=1$ if and only if $c^\frac{q+1}2=b$;
\item $\rank M_L=2$ if and only if $\begin{vmatrix}c&b\\b&c^q\end{vmatrix}\ne0$ and
\[\begin{vmatrix}c&b&a^{q^2}+a\\b&c^q&b^q\\a^{q^2}+a&b^q&c^{q^2}\end{vmatrix}=\begin{vmatrix}b&a^{q^2}+a&b^{q^3}\\c^q&b^q&a^{q^3}+a^q\\b^q&c^{q^2}&b^{q^2}\end{vmatrix}=0,\]
in which case $\epsilon=1$ if and only if
\[\begin{vmatrix}c&b\\b&c^q\end{vmatrix}^\frac{q+1}2=\begin{vmatrix}b&c^q\\a^{q^2}+a&b^q\end{vmatrix};\]
\item $\rank M_L=3$ if and only if
\[\begin{vmatrix}c&b&a^{q^2}+a\\b&c^q&b^q\\a^{q^2}+a&b^q&c^{q^2}\end{vmatrix}\ne0\]
and $\det M_L=0$, in which case $\epsilon=1$ if and only if
\[\begin{vmatrix}c&b&a^{q^2}+a\\b&c^q&b^q\\a^{q^2}+a&b^q&c^{q^2}\end{vmatrix}^\frac{q+1}2=\begin{vmatrix}b&c^q&b^q\\a^{q^2}+a&b^q&c^{q^2}\\b^{q^3}&a^{q^3}+a^q&b^{q^2}\end{vmatrix};\]
\item $\rank M_L=4$ if and only if $\det M_L\ne0$, in which case $\epsilon=-\eta(\det M_L)$.
\end{itemize}
\end{example}

\section{Quadratic forms represented by polynomials with coefficients in $\mathbb F_q$}\label{base}

For $L=\frac{\ell+\ell^\prime}2$, consider the special case that $L$ has all coefficients in $\mathbb F_q$. Some Artin-Schreier curves induced from this type of polynomials have been studied in, e.g., \cite[Theorem 5.3]{anbar2014quadratic} and \cite[Theorem 3.1]{oliveira2023artin}, with some restrictions such as $\gcd(n,q)=1$. In fact, the matrix $M_L$ is a circulant matrix that can be easily diagonalized, provided that there exists a primitive $n$-th root of unity over $\mathbb F_q$. However, this is not necessarily the case. In this section, we will address the problem, providing a complete characterization of all those quadratic forms.

For a polynomial $f$ in the polynomial ring $\mathbb C[z]$ written as
\[f(z)=\sum_{i=0}^{n-1}c_iz^i,\]
let $C_f$ be the circulant matrix defined by
\[C_f=\begin{pmatrix}c_0&c_1&\cdots&c_{n-1}\\c_{n-1}&c_0&\cdots&c_{n-2}\\\vdots&\vdots&\ddots&\vdots\\c_1&c_2&\cdots&c_0\end{pmatrix}.\]
The rank of $C_f$ and the determinant of some submatrix of it can be determined as follows.

\begin{lemma}\label{zeta}
Let $\zeta_1,\dots,\zeta_n$ be the distinct $n$-th roots of unity in $\mathbb C$. Let $f$ and $g$ be polynomials over $\mathbb C$ such that $\prod_{j=1}^rg(\zeta_j)\ne0$ and
\[f(z)=\sum_{i=0}^{n-1}c_iz^i\equiv g(z)\prod_{i=r+1}^n(z-\zeta_i)\pmod{z^n-1}\]
with $0<r\le n$. Then $\rank C_f=r$ and $\det C_f^{(r)}=\prod_{j=1}^r\zeta_j^{-r}g(\zeta_j)$.
\end{lemma}
\begin{proof}
Let $Z$ be the $n\times n$ matrix whose $(i,j)$ entry is $\zeta_j^{i-1}$. A simple investigation yields
\[Z^{-1}=\frac1n\begin{pmatrix}1&\zeta_1^{-1}&\cdots&\zeta_1^{-(n-1)}\\1&\zeta_2^{-1}&\cdots&\zeta_2^{-(n-1)}\\\vdots&\vdots&\ddots&\vdots\\1&\zeta_n^{-1}&\cdots&\zeta_n^{-(n-1)}\end{pmatrix}.\]
Note that
\[\begin{split}C_fZ&=\begin{pmatrix}c_0&c_1&\cdots&c_{n-1}\\c_{n-1}&c_0&\cdots&c_{n-2}\\\vdots&\vdots&\ddots&\vdots\\c_1&c_2&\cdots&c_0\end{pmatrix}\begin{pmatrix}1&1&\cdots&1\\\zeta_1&\zeta_2&\cdots&\zeta_n\\\vdots&\vdots&\ddots&\vdots\\\zeta_1^{n-1}&\zeta_2^{n-1}&\cdots&\zeta_n^{n-1}\end{pmatrix}\\&=\begin{pmatrix}f(\zeta_1)&f(\zeta_2)&\cdots&f(\zeta_n)\\\zeta_1f(\zeta_1)&\zeta_2f(\zeta_2)&\cdots&\zeta_nf(\zeta_n)\\\vdots&\vdots&\ddots&\vdots\\\zeta_1^{n-1}f(\zeta_1)&\zeta_2^{n-1}f(\zeta_2)&\cdots&\zeta_n^{n-1}f(\zeta_n)\end{pmatrix}\\&=\begin{pmatrix}1&1&\cdots&1\\\zeta_1&\zeta_2&\cdots&\zeta_n\\\vdots&\vdots&\ddots&\vdots\\\zeta_1^{n-1}&\zeta_2^{n-1}&\cdots&\zeta_n^{n-1}\end{pmatrix}\begin{pmatrix}f(\zeta_1)&&&\\&f(\zeta_2)&&\\&&\ddots&\\&&&f(\zeta_n)\end{pmatrix}\\&=ZD,\end{split}\]
where $D$ is the above diagonal matrix. This proves $\rank C_f=r$, as $\prod_{j=1}^rf(\zeta_j)=\prod_{j=1}^rg(\zeta_j)\ne0$.

Partition $D$ into four blocks where the upper-left one is $D^{(r)}$, so that the others are all zero. Since $C_f=ZDZ^{-1}$, partitioning $C_f$, $Z$ and $Z^{-1}$ in the same manner gives
\[\det C_f^{(r)}=\frac1{n^r}\begin{vmatrix}1&1&\cdots&1\\\zeta_1&\zeta_2&\cdots&\zeta_r\\\vdots&\vdots&\ddots&\vdots\\\zeta_1^{r-1}&\zeta_2^{r-1}&\cdots&\zeta_r^{r-1}\end{vmatrix}\begin{vmatrix}1&\zeta_1^{-1}&\cdots&\zeta_1^{-(r-1)}\\1&\zeta_2^{-1}&\cdots&\zeta_2^{-(r-1)}\\\vdots&\vdots&\ddots&\vdots\\1&\zeta_r^{-1}&\cdots&\zeta_r^{-(r-1)}\end{vmatrix}\det D^{(r)}.\]
Calculating the Vandermonde determinants, we get
\[\begin{split}n^r\det C_f^{(r)}&=\prod_{1\le j<i\le r}(\zeta_i-\zeta_j)\prod_{1\le j<i\le r}(\zeta_i^{-1}-\zeta_j^{-1})\prod_{j=1}^rf(\zeta_j)\\&=\prod_{1\le i<j\le r}(\zeta_j-\zeta_i)\prod_{1\le j<i\le r}(\zeta_j-\zeta_i)\zeta_i^{-1}\zeta_j^{-1}\prod_{j=1}^rf(\zeta_j)\\&=\prod_{j=1}^r\prod_{\substack{1\le i\le r\\i\ne j}}(\zeta_j-\zeta_i)\prod_{1\le j<i\le r}\zeta_i^{-1}\zeta_j^{-1}\prod_{j=1}^r\Bigg(g(\zeta_j)\prod_{i=r+1}^n(\zeta_j-\zeta_i)\Bigg)\\&=\prod_{j=1}^r\prod_{\substack{1\le i\le n\\i\ne j}}(\zeta_j-\zeta_i)\prod_{j=1}^r\zeta_j^{1-r}\prod_{j=1}^rg(\zeta_j)\\&=\prod_{j=1}^rn\zeta_j^{-1}\prod_{j=1}^r\zeta_j^{1-r}g(\zeta_j)\\&=n^r\prod_{j=1}^r\zeta_j^{-r}g(\zeta_j),\end{split}\]
where the equality
\[\prod_{\substack{1\le i\le n\\i\ne j}}(\zeta_j-\zeta_i)=n\zeta_j^{-1}\]
can be obtained by taking formal derivatives of both sides of
\[(z-\zeta_j)\prod_{\substack{1\le i\le n\\i\ne j}}(z-\zeta_i)=z^n-1\]
at $\zeta_j$.
\end{proof}

\begin{remark}
In this lemma, $\prod_{j=1}^r\zeta_j^{-r}g(\zeta_j)$ can be calculated from the coefficients of $f$ using the resultant of $\frac{f(z)}{\gcd(f(z),z^n-1)}$ and $\frac{z^n-1}{\gcd(f(z),z^n-1)}$.
\end{remark}

For $L(x)=\sum_{i=0}^{n-1}a_ix^{q^i}$ with all coefficients in $\mathbb F_q$, define
\[L^*(x)=\sum_{i=0}^{n-1}a_ix^i.\]
Here is the finite field version of the preceding result (for fields of even characteristic as well).

\begin{theorem}
Let $\omega_1,\dots,\omega_n$ be the (not necessarily distinct) $n$-th roots of unity over $\mathbb F_q$ such that $x^n-1=\prod_{i=1}^n(x-\omega_i)$. Let $L$ be a nonzero $q$-linear polynomial over $\mathbb F_q$ of degree less than $q^n$ and
\[\gcd(L^*(x),x^n-1)=\prod_{i=r+1}^n(x-\omega_i)\]
with $0<r\le n$, so that
\[L^*(x)\equiv G(x)\prod_{i=r+1}^n(x-\omega_i)\pmod{x^n-1}\]
for some polynomial $G$ over $\mathbb F_q$ (for instance, let $G(x)=\frac{L^*(x)}{\gcd(L^*(x),x^n-1)}$). Then $\rank M_L=r$, and $\det M_L^{(r)}=\prod_{j=1}^r\omega_j^{-r}G(\omega_j)$. In particular, if $r=n$, then $\det M_L=\prod_{j=1}^nL^*(\omega_j)$.
\end{theorem}
\begin{proof}
Let $n=p^e\nu$ with $\gcd(\nu,p)=1$, where $p$ is the characteristic of $\mathbb F_q$. Denote by $K$ the $(q-1)$-st cyclotomic field with ring of integers $\mathcal O_K$, and let $E=K(\zeta)$ for a primitive $n$-th root of unity $\zeta$ in $\mathbb C$. The ring of integers of $E$ is $\mathcal O_E=\mathcal O_K[\zeta]$. For a prime $P$ of $K$ lying above $p\mathbb Z$, the residue field $\mathcal O_K/P$ has $q$ elements and there is a surjective ring homomorphism $\sigma$ from $\mathcal O_K$ to $\mathbb F_q$ with kernel $P$. Take the minimal polynomial of $\zeta$ over $K$, which has coefficients all in $\mathcal O_K$, and apply $\sigma$ to each of these coefficients to get a polynomial over $\mathbb F_q$. Let $\omega$ be a root of this polynomial, so that $\mathbb F_q(\omega)$ is a fintie extension of $\mathbb F_q$. By setting $\sigma(\zeta)=\omega$, we extend $\sigma$ to a well-defined homomorphism (still denoted by $\sigma$) from $\mathcal O_K[\zeta]$ onto $\mathbb F_q(\omega)$, whose kernel is a prime $\mathfrak P$ of $E$ containing $P$. Note that $\zeta^{p^e}$ is a primitive $\nu$-th root of unity and
\[\prod_{i=1}^{\nu-1}\big(1-\zeta^{p^ei}\big)=\nu\notin\mathfrak P,\]
for otherwise $\nu\in\mathfrak P\cap\mathbb Z=p\mathbb Z$. Then $\omega^i\ne1$ for $1\le i<\nu$ since $(1-\omega^i)^{p^e}=\sigma\big(1-\zeta^{p^ei}\big)\ne0$. Furthermore, $\omega^\nu=1$ as $(\omega^{\nu})^{p^e}=\sigma(\zeta)^n=1$, so $\omega$ is a primitive $\nu$-th root of unity over $\mathbb F_q$. We conclude that $\sigma$ maps the group of $n$-th roots of unity in $\mathbb C$ onto the group of $\nu$-th roots of unity over $\mathbb F_q$, and thus
\[x^n-1=\prod_{i=1}^n(x-\sigma(\zeta_i)),\]
where $\zeta_1,\dots,\zeta_n$ are distinct $n$-th roots of unity in $\mathbb C$ such that $\sigma(\zeta_i)=\omega_i$ for $1\le i\le n$.

In what follows, for a polynomial $f$ over $\mathcal O_K$, denote by $\bar f$ the polynomial over $\mathbb F_q$ obtained by applying $\sigma$ to each coefficient of $f$. Then $G=\bar g$ for some polynomials $g$ over $\mathcal O_K$, and
\[L^*(x)=\bar g(x)\prod_{i=r+1}^n(x-\omega_i)+\bar h(x)(x^n-1)\]
for some polynomials $h$ over $\mathcal O_K$, as $x^n-1$ divides $L^*(x)-G(x)\prod_{i=r+1}^n(x-\omega_i)$ in $\mathbb F_q[x]$. Let
\[f(z)=g(z)\prod_{i=r+1}^n(z-\zeta_i)+h(z)(z^n-1),\]
so that $\bar f=L^*$. Here $\prod_{j=1}^rg(\zeta_j)\ne0$; otherwise $\prod_{j=1}^r\bar g(\omega_j)=0$, but
\[\begin{split}\gcd(L^*(x),x^n-1)&=\gcd\Bigg(\bar g(x)\prod_{i=r+1}^n(x-\omega_i),x^n-1\Bigg)\\&=\gcd\Bigg(\bar g(x),\prod_{i=1}^r(x-\omega_i)\Bigg)\prod_{i=r+1}^n(x-\omega_i).\end{split}\]
It follows from Lemma \ref{zeta} that
\[\det C_f^{(r)}=\prod_{j=1}^r\zeta_j^{-r}g(\zeta_j).\]
Calculating the determinants of matrices by definition, one gets
\[\det M_L^{(r)}=\sigma(\det C_f^{(r)})=\prod_{j=1}^r\sigma(\zeta_j^{-r}g(\zeta_j))=\prod_{j=1}^r\omega_j^{-r}\bar g(\omega_j),\]
and
\[\det M_L=\prod_{j=1}^n\omega_j^{-n}G(\omega_j)=\prod_{j=1}^nL^*(\omega_j)\]
when $r=n$. Moreover, every singular submatrix (of zero determinant) of $C_f$ induces a singular submatrix of $M_L$, so $\rank M_L\le\rank C_f=r$. This implies $\rank M_L=r$.
\end{proof}

\begin{theorem}\label{e}
Continue with the notation in the last theorem, and suppose that $L=\frac{\ell+\ell^\prime}2$ be nonzero with all coefficients in $\mathbb F_q$. Then
\begin{enumerate}
\item\label{e_1}$\epsilon=1$ if and only if $\prod_{j=1}^rG(\omega_j)^\frac{q-1}2=(-1)^{r-1}c^{\frac{q-1}2r+1}$, where $c$ is the constant term of $\gcd(L^*(x),x^n-1)$ satisfying $c^2=1$;
\item\label{e_2}$\epsilon=(-1)^{(n-1-\mu)(\frac{q-1}2r+1)}\eta\big(\prod_{j=1}^rG(\omega_j)\big)$,
where $\mu$ is the multiplicity of $-1$ as a root of $\gcd(L^*(x),x^n-1)$;
\item $\epsilon=(-1)^{n-1}\eta\big(\prod_{j=1}^nL^*(\omega_j)\big)$, when $r=n$;
\item\label{e_4}$\epsilon=(-1)^{(n-1)(\frac{q-1}2m+1)}\eta(2a)^n$, when $r=n-2m<n$ and $\ell$ has degree $q^m$ with leading coefficient $a$.
\end{enumerate}
\end{theorem}
\begin{proof}
The first three statements are apparent in the case $r=n$. Suppose that $0<r<n$ and let
\[L^*(x)=G(x)\prod_{i=r+1}^n(x-\omega_i)+H(x)(x^n-1)\]
for some polynomial $H$ over $\mathbb F_q$. Let $k$ be a sufficiently large integer such that $x^{k-n+r}G(x^{-1})$ and $x^{k-n}H(x^{-1})$ are both polynomials in $\mathbb F_q[x]$. Then
\[\begin{split}x^kL^*(x^{-1})&=x^kG(x^{-1})\prod_{i=r+1}^n(x^{-1}-\omega_i)+x^kH(x^{-1})(x^{-n}-1)\\&=x^{k-n+r}G(x^{-1})\prod_{i=r+1}^n\omega_i(\omega_i^{-1}-x)+x^{k-n}H(x^{-1})(1-x^n),\end{split}\]
which means
\begin{equation}\label{eq9}x^kL^*(x^{-1})\equiv x^{k-n+r}G(x^{-1})\prod_{i=r+1}^n\omega_i(\omega_i^{-1}-x)\pmod{x^n-1}.\end{equation}
For $L(x)=\sum_{i=0}^{n-1}a_ix^{q^i}$, one has
\[x^kL^*(x^{-1})=\sum_{i=0}^{n-1}a_ix^{k-i}\equiv\sum_{i=0}^{n-1}a_ix^{k+n-i}\equiv x^kL^{\prime*}(x)\pmod{x^n-1},\]
so
\[\gcd(x^kL^*(x^{-1}),x^n-1)=\gcd(x^kL^*(x),x^n-1)=\prod_{i=r+1}^n(x-\omega_i)\]
as $L=L^\prime$. It follows from \eqref{eq9} that $\prod_{i=r+1}^n(x-\omega_i^{-1})$ divides $x^kL^*(x^{-1})$, and consequently it divides $\prod_{i=r+1}^n(x-\omega_i)$. Then
\[\prod_{i=r+1}^n(x-\omega_i)=\prod_{i=r+1}^n(x-\omega_i^{-1}),\]
where $\omega_i$ and $\omega_i^{-1}$ have the same multiplicity. Hence, $(-1)^\mu\prod_{j=1}^r\omega_j=\prod_{j=1}^n\omega_j=(-1)^{n-1}$. Moreover,
\[c=(-1)^{n-r}\prod_{i=r+1}^n\omega_i=(-1)^{n-r}\prod_{i=r+1}^n\omega_i^{-1}=c^{-1},\]
and thus $\prod_{j=1}^r\omega_j=(-1)^{r-1}c$ with $c^2=1$.

Let $L_1$ be the polynomial of degree less than $q^n$ with
\[L_1(x)\equiv L(x)^q\pmod{x^{q^n}-x},\]
so that
\[L_1^*(x)\equiv xL^*(x)\equiv xG(x)\prod_{i=r+1}^n(x-\omega_i)\pmod{x^n-1}.\]
Since
\[\gcd(L_1^*(x),x^n-1)=\gcd(xL^*(x),x^n-1)=\prod_{i=r+1}^n(x-\omega_i),\]
the last theorem can be applied to get
\[\det M_{L_1}^{(r)}=\prod_{j=1}^r\omega_j^{1-r}G(\omega_j).\]
Observe that the submatrix of $M_L$ defined in \ref{det_1} of Theorem \ref{det} is, in this case, exactly $M_{L_1}^{(r)}$. Then $\epsilon=1$ if and only if
\[\prod_{j=1}^r(\omega_j^{-r}G(\omega_j))^\frac{q+1}2=\prod_{j=1}^r\omega_j^{1-r}G(\omega_j);\]
that is,
\[\prod_{j=1}^rG(\omega_j)^\frac{q-1}2=\prod_{j=1}^r\omega_j^{\frac{q-1}2r+1}.\]
The results \ref{e_1} and \ref{e_2} can be easily verified.

Finally, by the assumption of \ref{e_4} we have
\[x^mL^*(x)\equiv\frac a2x^{2m}+\dots+\frac a2\pmod{x^n-1},\]
where the polynomial on the right side has degree $2m$, while
\[\gcd(x^mL^*(x),x^n-1)=\gcd(L^*(x),x^n-1)=\prod_{i=r+1}^n(x-\omega_i)\]
has degree $n-r=2m$. Therefore,
\[x^{2m}+\dots+1=\prod_{i=r+1}^n(x-\omega_i),\]
which implies $\prod_{j=1}^r\omega_j=(-1)^{r-1}c=(-1)^{n-1}$ and
\[L^*(x)\equiv x^{n-m}x^mL^*(x)\equiv\frac a2x^{n-m}\prod_{i=r+1}^n(x-\omega_i)\pmod{x^n-1}.\]
Then
\[\begin{split}\epsilon&=(-1)^{n-1}\eta\Big(\det M_L^{(r)}\Big)\\&=(-1)^{n-1}\eta\Bigg(\prod_{j=1}^r\frac a2\omega_j^m\Bigg)\\&=(-1)^{n-1}\eta\Big((-1)^{(n-1)m}\Big(\frac a2\Big)^r\Big)\\&=(-1)^{(n-1)(\frac{q-1}2m+1)}\eta(2a)^n,\end{split}\]
as a consequence of \ref{det_3} in Theorem \ref{det}.
\end{proof}

\section{Maximal and minimal Artin-Schreier curves from some quadratic forms}\label{curve}

To find a maximal or minimal curve as in Lemma \ref{maximal}, it is necessary and sufficient to let $n$ be even and $\ell$ have degree $q^m$ such that $\dim_{\mathbb F_q}\ker L=2m$ for $L=\frac{\ell+\ell^\prime}2$. The trivial cases include
\begin{itemize}
\item $\ell(x)=ax^{q^\frac n2}$ with $a\in\mathbb F_{q^n}^*$ such that $a^{q^\frac n2}+a=0$;
\item $\ell(x)=ax$ with $a\in\mathbb F_{q^n}^*$.
\end{itemize}

Using the obtained classification results of quadratic forms, we can derive more such curves explicitly and decide whether they are maximal or minimal, which depends on $(-1)^\frac{(q-1)(n-2m)}4\epsilon$. The case of odd $n$ will be considered as well.

\begin{example}
Consider $\ell(x)=2ax^{q^2}+2bx^q+cx$ with $a\in\mathbb F_q^*$ and $b,c\in\mathbb F_q$, and suppose that $n>4$. For $L=\frac{\ell+\ell^\prime}2$, we have
\[x^2L^*(x)\equiv ax^4+bx^3+cx^2+bx+a\pmod{x^n-1},\]
and $\dim_{\mathbb F_q}\ker L=4$ if and only if $\gcd(ax^4+bx^3+cx^2+bx+a,x^n-1)$ has degree $4$. For an $n$-th root of unity $\omega$ over $\mathbb F_q$, it is clear that $\omega$ and $\omega^{-1}$ have the same multiplicity as roots of $ax^4+bx^3+cx^2+bx+a$. In addition, the multiplicity of $-1$, as well as of $1$, must be even. In order for $a^{-1}(ax^4+bx^3+cx^2+bx+a)$ to divide $x^n-1$, it is necessarily
\begin{itemize}
\item  the product of two polynomials of the form $(x-\omega)(x-\omega^{-1})$ with $\omega\in\mathbb F_q$ or $\omega^{q+1}=1$, or
\item $(x-\omega)(x-\omega^q)(x-\omega^{-1})(x-\omega^{-q})$ with $\omega\in\mathbb F_{q^2}\setminus\mathbb F_q$ such that $\omega^{q+1}\ne1$, or
\item $(x-\omega)(x-\omega^q)\big(x-\omega^{q^2}\big)\big(x-\omega^{q^3}\big)$ with $\omega\in\mathbb F_{q^4}\setminus\mathbb F_{q^2}$.
\end{itemize}
The sufficiency is routine to check. If $\dim_{\mathbb F_q}\ker L=4$, then $\epsilon=(-1)^{n-1}\eta(a)^n$ by \ref{e_4} of Theorem \ref{e}.
\end{example}

Next, we investigate certain binomials that may induce maximal or minimal Artin-Schreier curves.

\begin{proposition}\label{L}
Let $L(x)=(ax)^{q^{n-1}}+ax^q+bx$ for $a,b\in\mathbb F_{q^n}^*$ with $n\ge4$. Define a sequence of polynomials in $\mathbb F_{q^n}[x]$ as
\[f_k(x)=-f_{k-1}(x)-x^{q^{k-2}}f_{k-2}(x)\quad\text{for }k\ge2,\]
with $f_0=1$ and $f_1=-1$. Then $\dim_{\mathbb F_q}\ker L=2$ if and only if
\[f_{n-1}(a^2b^{-1-q})=\mathrm N(a)-\mathrm N(b)f_n(a^2b^{-1-q})=0,\]
if and only if
\[a^{2q^{n-3}}\det M_L^{(n-3)}-b^{q^{n-2}}\det M_L^{(n-2)}=\mathrm N(a)+(-1)^na^{2q^{n-2}}\det M_L^{(n-2)}=0.\]
\end{proposition}
\begin{proof}
The kernel of $L$ has at most $q^2$ elements, since
\[L(x)^q\equiv a^qx^{q^2}+b^qx^q+ax\pmod{x^{q^n}-x}.\]
By Theorem \ref{rank}, $\dim_{\mathbb F_q}\ker L=2$ if and only if
\[\begin{pmatrix}\beta_0&\cdots&\beta_{n-2}&0\end{pmatrix}\begin{pmatrix}b&a&&&&a^{q^{n-1}}\\a&b^q&\ddots\\&a^q&\ddots&a^{q^{n-4}}\\&&\ddots&b^{q^{n-3}}&a^{q^{n-3}}\\&&&a^{q^{n-3}}&b^{q^{n-2}}&a^{q^{n-2}}\\a^{q^{n-1}}&&&&a^{q^{n-2}}&b^{q^{n-1}}\end{pmatrix}\]
is zero for some $\beta_0,\dots,\beta_{n-2}\in\mathbb F_{q^n}$ one of which is nonzero. Assume that this is the case. If $\beta_0=0$, then clearly $\beta_1,\dots,\beta_{n-2}$ are all zero. Hence, let $\beta_0=1$ without loss of generality. Now we have
\[\begin{cases}b+a\beta_1=0,\\a^{q^{k-2}}\beta_{k-2}+b^{q^{k-1}}\beta_{k-1}+a^{q^{k-1}}\beta_k=0&\text{for }2\le k\le n-2,\\a^{q^{n-3}}\beta_{n-3}+b^{q^{n-2}}\beta_{n-2}=0,\\a^{q^{n-1}}+a^{q^{n-2}}\beta_{n-2}=0,\end{cases}\]
where $\beta_1,\dots,\beta_{n-2}$ are uniquely determined by all but the last two equations. We conclude that $\dim_{\mathbb F_q}\ker L=2$ only if for the sequence of elements in $\mathbb F_{q^n}$ defined by
\[\beta_k=-a^{-q^{k-1}}\big(b^{q^{k-1}}\beta_{k-1}+a^{q^{k-2}}\beta_{k-2}\big)\quad\text{for }k\ge2,\]
with $\beta_0=1$ and $\beta_1=-a^{-1}b$, one has
\begin{equation}\label{eq0}a^{q^{n-3}}\beta_{n-3}+b^{q^{n-2}}\beta_{n-2}=a^{q^{n-1}}+a^{q^{n-2}}\beta_{n-2}=0.\end{equation}
The converse is apparent.

We show by induction that with $\beta_k$ defined above,
\begin{equation}\label{eq1}a^{q^{k-1}+\dots+q+1}\beta_k=b^{q^{k-1}+\dots+q+1}f_k(a^2b^{-1-q})=(-1)^k\det M_L^{(k)}\end{equation}
for $1\le k\le n-2$. For $k\in\{1,2\}$, the claim is obvious. Let $3\le k\le n-2$ and assume that it holds for $k-1$ and $k-2$. Then
\[\begin{split}&\mathrel{\phantom{=}}a^{q^{k-1}+\dots+q+1}\beta_k\\&=-a^{q^{k-2}+\dots+q+1}\big(b^{q^{k-1}}\beta_{k-1}+a^{q^{k-2}}\beta_{k-2}\big)\\&=-b^{q^{k-1}}a^{q^{k-2}+\dots+q+1}\beta_{k-1}-a^{2q^{k-2}}a^{q^{k-3}+\dots+q+1}\beta_{k-2}\\&=-b^{q^{k-1}+\dots+q+1}f_{k-1}(a^2b^{-1-q})-a^{2q^{k-2}}b^{q^{k-3}+\dots+q+1}f_{k-2}(a^2b^{-1-q})\\&=b^{q^{k-1}+\dots+q+1}\big(-f_{k-1}(a^2b^{-1-q})-(a^2b^{-1-q})^{q^{k-2}}f_{k-2}(a^2b^{-1-q})\big)\\&=b^{q^{k-1}+\dots+q+1}f_k(a^2b^{-1-q}),\end{split}\]
and
\[\begin{split}&\mathrel{\phantom{=}}a^{q^{k-1}+\dots+q+1}\beta_k\\&=-b^{q^{k-1}}a^{q^{k-2}+\dots+q+1}\beta_{k-1}-a^{2q^{k-2}}a^{q^{k-3}+\dots+q+1}\beta_{k-2}\\&=(-1)^kb^{q^{k-1}}\det M_L^{(k-1)}-(-1)^ka^{2q^{k-2}}\det M_L^{(k-2)}\\&=(-1)^k\det M_L^{(k)},\end{split}\]
where the last equality follows from
\[\begin{split}&\mathrel{\phantom{=}}\begin{vmatrix}b&a\\a&b^q&\ddots\\&a^q&\ddots&a^{q^{k-3}}\\&&\ddots&b^{q^{k-2}}&a^{q^{k-2}}\\&&&a^{q^{k-2}}&b^{q^{k-1}}\end{vmatrix}\\&=b^{q^{k-1}}\det M_L^{(k-1)}-a^{q^{k-2}}\begin{vmatrix}b&a\\a&b^q&\ddots\\&a^q&\ddots&a^{q^{k-4}}\\&&\ddots&b^{q^{k-3}}\\&&&a^{q^{k-3}}&a^{q^{k-2}}\end{vmatrix}\\&=b^{q^{k-1}}\det M_L^{(k-1)}-a^{2q^{k-2}}\det M_L^{(k-2)}.\end{split}\]

Finally, by \eqref{eq1} we have
\[\begin{split}&\mathrel{\phantom{=}}a^{q^{n-4}+\dots+q+1}\big(a^{q^{n-3}}\beta_{n-3}+b^{q^{n-2}}\beta_{n-2}\big)\\&=a^{q^{n-3}}b^{q^{n-4}+\dots+q+1}f_{n-3}(a^2b^{-1-q})+a^{-q^{n-3}}b^{q^{n-2}+\dots+q+1}f_{n-2}(a^2b^{-1-q})\\&=a^{-q^{n-3}}b^{q^{n-2}+\dots+q+1}\big((a^2b^{-1-q})^{q^{n-3}}f_{n-3}(a^2b^{-1-q})+f_{n-2}(a^2b^{-1-q})\big)\\&=-a^{-q^{n-3}}b^{q^{n-2}+\dots+q+1}f_{n-1}(a^2b^{-1-q}),\end{split}\]
and
\[\begin{split}&\mathrel{\phantom{=}}a^{q^{n-3}+\dots+q+1}\big(a^{q^{n-3}}\beta_{n-3}+b^{q^{n-2}}\beta_{n-2}\big)\\&=(-1)^{n-1}a^{2q^{n-3}}\det M_L^{(n-3)}+(-1)^nb^{q^{n-2}}\det M_L^{(n-2)},\end{split}\]
with
\[\begin{split}&\mathrel{\phantom{=}}a^{q^{n-2}+\dots+q+1}\big(a^{q^{n-1}}+a^{q^{n-2}}\beta_{n-2}\big)\\&=\mathrm N(a)+\mathrm N(b)(a^2b^{-1-q})^{q^{n-2}}f_{n-2}(a^2b^{-1-q})\\&=\mathrm N(a)+(-1)^na^{2q^{n-2}}\det M_L^{(n-2)}.\end{split}\]
Given the condition in \eqref{eq0}, this completes the proof.
\end{proof}

\begin{corollary}
Continue with the notation in the last proposition, and suppose that $a^2b^{-1-q}\in\mathbb F_q$. Define a sequence of polynomials in $\mathbb F_q[x]$ as
\[g_k(x)=(-1)^k\sum_{i=0}^{\lfloor\frac k2\rfloor}(-1)^i\binom{k-i}ix^i\]
for $k\ge0$. Then $\dim_{\mathbb F_q}\ker L=2$ if and only if
\[g_{n-1}(a^2b^{-1-q})=\mathrm N(a)-\mathrm N(b)g_n(a^2b^{-1-q})=0.\]
\end{corollary}
\begin{proof}
For $k\ge2$, one has
\[\begin{split}&\mathrel{\phantom{=}}(-1)^{k-1}(g_{k-1}(x)+xg_{k-2}(x))\\&=\sum_{i=0}^{\lfloor\frac{k-1}2\rfloor}(-1)^i\binom{k-1-i}ix^i-\sum_{i=0}^{\lfloor\frac{k-2}2\rfloor}(-1)^i\binom{k-2-i}ix^{i+1}\\&=1+\sum_{i=1}^{\lfloor\frac{k-1}2\rfloor}(-1)^i\binom{k-1-i}ix^i+\sum_{i=1}^{\lfloor\frac k2\rfloor}(-1)^i\binom{k-1-i}{i-1}x^i\\&=1+\sum_{i=1}^{\lfloor\frac{k-1}2\rfloor}(-1)^i\binom{k-i}ix^i+\sum_{i=\lfloor\frac{k-1}2\rfloor+1}^{\lfloor\frac k2\rfloor}(-1)^i\binom{k-1-i}{i-1}x^i\\&=(-1)^kg_k(x)-\sum_{i=\lfloor\frac{k-1}2\rfloor+1}^{\lfloor\frac k2\rfloor}\bigg((-1)^i\binom{k-i}ix^i-(-1)^i\binom{k-1-i}{i-1}x^i\bigg).\end{split}\]
If $k$ is odd, then $\lfloor\frac{k-1}2\rfloor=\lfloor\frac k2\rfloor$; otherwise, $\lfloor\frac{k-1}2\rfloor+1=\lfloor\frac k2\rfloor=\frac k2$ and
\[\binom{k-\frac k2}{\frac k2}-\binom{k-1-\frac k2}{\frac k2-1}=1-1=0.\]
Consequently,
\[g_k(x)=-g_{k-1}(x)-xg_{k-2}(x)\quad\text{for }k\ge2\]
with $g_0=1$ and $g_1=-1$ by definition. Note that
\[f_k(x)=-f_{k-1}(x)-x^{q^{k-2}}f_{k-2}(x)\equiv-f_{k-1}(x)-xf_{k-2}(x)\pmod{x^q-x}\]
for $k\ge2$ with $f_0=1$ and $f_1=-1$. Therefore,
\[f_k(x)\equiv g_k(x)\pmod{x^q-x}\]
and $f_k(a^2b^{-1-q})=g_k(a^2b^{-1-q})$.
% \[x_0,x_1=\frac{-1\pm\sqrt{1-4x}}2\]
% \[f_i=x_0^{i+1}/s-x_1^{i+1}/s.\]
\end{proof}

\begin{corollary}
Let $\ell(x)=2ax^q+bx$ and $L(x)=(ax)^{q^{n-1}}+ax^q+bx$ for $a,b\in\mathbb F_{q^n}^*$ with $n\ge4$. If $\dim_{\mathbb F_q}\ker L=2$, then $\epsilon=(-1)^{\frac{q+1}2(n-1)}\eta(\mathrm N(a))$.
\end{corollary}
\begin{proof}
It follows from Proposition \ref{L} that
$\mathrm N(a)+(-1)^na^{2q^{n-2}}\det M_L^{(n-2)}=0$. Therefore,
\[\epsilon=(-1)^{n-1}\eta((-1)^{n-1}\mathrm N(a))=(-1)^{\frac{q+1}2(n-1)}\eta(\mathrm N(a)),\]
as a result of \ref{det_3} of Theorem \ref{det}.
\end{proof}

Consider the binomial $\ell(x)=ax^{q^m}+bx^{q^l}$ for $a,b\in\mathbb F_{q^n}^*$ with positive integers $m$ and $l$. It is easy to see that $\dim_{\mathbb F_q}\ker\ell$ is either $0$ or $\gcd(m-l,n)$. Denote by $v_2$ the $2$-adic valuation of $\mathbb Q$, and by $\mathrm{Tr}_k$ the trace map of $\mathbb F_{q^n}/\mathbb F_{q^k}$ for a positive divisor $k$ of $n$. We have the following results concerning the composition of two linear endomorphisms of $\mathbb F_{q^n}/\mathbb F_q$.

\begin{theorem}
Let $\ell(x)=ax^{q^m}+bx^{q^l}$ for $a,b\in\mathbb F_{q^n}^*$ with $0<l<m<\frac n2$, $d=\gcd(m,m-l)$ and $e=\lcm(m+l,m-l)$. Then
\begin{enumerate}
\item\label{b_1}$\alpha\ell(x)^{q^{n-m-l}}=\ell^\prime(x)$ for some $\alpha\in\mathbb F_{q^n}^*$ if and only if $a^{q^l+1}=b^{q^m+1}$;
\item\label{b_2}if both $m+l$ and $m-l$ divide $n$, $a=\gamma^{q^m+1}\delta$ and $b=-\gamma^{q^l+1}\delta$ for $\gamma\in\mathbb F_{q^n}^*$ and $\delta\in\mathbb F_{q^{m-l}}^*$ such that
\begin{itemize}
\item $v_2(l)\ge v_2(m-l)$ and $\gcd\big(\frac ne,q\big)\ne1$, or
\item $v_2(l)<v_2(m-l)$ and $\delta^{\frac{q^{m-l}-1}{q^d+1}\gcd(\frac ne,q^d+1)}=1$ with either $\delta^{\frac{q^{m-l}-1}{q^d+1}}\ne1$ or $\gcd\big(\frac ne,q\big)\ne1$,
\end{itemize}
then $\dim_{\mathbb F_q}\ker L=2m$ for $L=\frac{\ell+\ell^\prime}2$;
\item\label{b_3}the converse of \ref{b_2} is true provided that $a^{q^l+1}=b^{q^m+1}$.
\end{enumerate}
\end{theorem}
\begin{proof}
First of all, we claim that
\begin{itemize}
\item if $v_2(l)\ge v_2(m-l)$, then $(q^m-1)\frac{q^e-1}{q^{m+l}-1}$ is divisible by $q^{m-l}-1$;
\item if $v_2(l)<v_2(m-l)$, then
\[\gcd\Big((q^m-1)\frac{q^e-1}{q^{m+l}-1},q^{m-l}-1\Big)=\frac{q^{m-l}-1}{q^d+1}\]
and
\[\gcd\Big((q^m-1)\frac{q^e-1}{q^{m+l}-1}\frac ne,q^{m-l}-1\Big)=\frac{q^{m-l}-1}{q^d+1}\gcd\Big(\frac ne,q^d+1\Big),\]
when $n$ is divisible by $e$.
\end{itemize}
For $d=\gcd(m,m-l)=\gcd(l,m-l)$, if $v_2(l)\ge v_2(m-l)$, then $\gcd(m+l,m-l)=\gcd(2l,m-l)=d$ and
\[(q^m-1)\frac{q^e-1}{q^{m+l}-1}=\frac{q^m-1}{q^d-1}\frac{(q^e-1)(q^d-1)}{q^{m+l}-1}\]
is divisible by $q^{m-l}-1$. To see this, note that $(z^{m+l}-1)(z^{m-l}-1)$ divides $(z^e-1)(z^d-1)$ in $\mathbb C[z]$; both of them are monic with coefficients in $\mathbb Z$, which means $\frac{(z^e-1)(z^d-1)}{(z^{m+l}-1)(z^{m-l}-1)}$ is a polynomial over $\mathbb Z$. Assume that $v_2(l)<v_2(m-l)$. Then $\gcd(m+l,m-l)=2d$ and $\big(z^\frac{m+l}{2d}-1\big)\big(z^\frac{m-l}{2d}-1\big)$ divides $\big(z^\frac e{2d}-1\big)(z-1)$ over $\mathbb Z$. Moreover,
\[\frac{\big(z^\frac e{2d}-1\big)(z-1)}{\big(z^\frac{m+l}{2d}-1\big)\big(z^\frac{m-l}{2d}-1\big)}=\frac{\frac{z^\frac e{2d}-1}{z-1}}{\frac{z^\frac{m+l}{2d}-1}{z-1}\frac{z^\frac{m-l}{2d}-1}{z-1}}\equiv\frac{\frac e{2d}}{\frac{m+l}{2d}\frac{m-l}{2d}}\equiv1\pmod{z-1},\]
with $\frac{\big(z^\frac e{2d}-1\big)(z-1)}{\big(z^\frac{m+l}{2d}-1\big)\big(z^\frac{m-l}{2d}-1\big)}-1$ divisible by $z-1$ over $\mathbb Z$. Substitute $q^{2d}$ for $z$ to get
\begin{equation}\label{eq4}\frac{(q^e-1)(q^{2d}-1)}{(q^{m+l}-1)(q^{m-l}-1)}\equiv1\pmod{q^{2d}-1}.\end{equation}
Now
\[(q^m-1)\frac{q^e-1}{q^{m+l}-1}=\frac{q^{m-l}-1}{q^d+1}\frac{q^m-1}{q^d-1}\frac{(q^e-1)(q^{2d}-1)}{(q^{m+l}-1)(q^{m-l}-1)}.\]
Noting that $v_2(m)=v_2(l+m-l)=\min\{v_2(l),v_2(m-l)\}=v_2(d)$ and
\[\frac{q^m-1}{q^d-1}=\frac{(q^d)^\frac md-1}{q^d-1}=\sum_{i=0}^{\frac md-1}q^{di}\equiv1\pmod{q^d+1},\]
with
\[\frac{(q^e-1)(q^{2d}-1)}{(q^{m+l}-1)(q^{m-l}-1)}\equiv1\pmod{q^d+1}\]
from \eqref{eq4}, we get
\[\begin{split}&\mathrel{\phantom{=}}\gcd\Big((q^m-1)\frac{q^e-1}{q^{m+l}-1},q^{m-l}-1\Big)\\&=\frac{q^{m-l}-1}{q^d+1}\gcd\Big(\frac{q^m-1}{q^d-1}\frac{(q^e-1)(q^{2d}-1)}{(q^{m+l}-1)(q^{m-l}-1)},q^d+1\Big)\\&=\frac{q^{m-l}-1}{q^d+1}.\end{split}\]
A similar argument gives
\[\gcd\Big((q^m-1)\frac{q^e-1}{q^{m+l}-1}\frac ne,q^{m-l}-1\Big)=\frac{q^{m-l}-1}{q^d+1}\gcd\Big(\frac ne,q^d+1\Big),\]
when $e$ divides $n$. The claim is now valid.

For $\alpha\in\mathbb F_{q^n}$, we have
\[\alpha\ell(x)^{q^{n-m-l}}=\alpha a^{q^{n-m-l}}x^{q^{n-l}}+\alpha b^{q^{n-m-l}}x^{q^{n-m}},\]
and
\[\ell^\prime(x)=(ax)^{q^{n-m}}+(bx)^{q^{n-l}}.\]
Then $\alpha\ell(x)^{q^{n-m-l}}=\ell^\prime(x)$ if and only if
\[\alpha a^{q^{n-m-l}}-b^{q^{n-l}}=\alpha b^{q^{n-m-l}}-a^{q^{n-m}}=0,\]
which is equivalent to
\begin{equation}\label{eq2}\alpha^{q^{m+l}}=b^{q^m}a^{-1}=a^{q^l}b^{-1},\end{equation}
and \ref{b_1} follows. Suppose that this is the case. Let $\ell_0(x)=\alpha x^{q^{n-m-l}}+x$ so that
\begin{equation}\label{eq6}\begin{split}\dim_{\mathbb F_q}\ker(\ell_0\circ\ell)&=\dim_{\mathbb F_q}(\ker\ell_0\cap\im\ell)+\dim_{\mathbb F_q}\ker\ell\\&\le\dim_{\mathbb F_q}\ker\ell_0+\dim_{\mathbb F_q}\ker\ell\\&\le\gcd(m+l,n)+\gcd(m-l,n)\\&\le m+l+m-l\\&=2m.\end{split}\end{equation}
Here, all the equalities hold if and only if $\ker\ell_0\cap\im\ell=\ker\ell_0$, $\dim_{\mathbb F_q}\ker\ell_0=m+l$ and $\dim_{\mathbb F_q}\ker\ell=m-l$ with both $m+l$ and $m-l$ dividing $n$.

Next, we show \ref{b_3} with the assumption that $\alpha\ell(x)^{q^{n-m-l}}=\ell^\prime(x)$ and $\dim_{\mathbb F_q}\ker(\ell_0\circ\ell)=\dim_{\mathbb F_q}\ker L=2m$. Given $\dim_{\mathbb F_q}\ker\ell_0$, we see that $\alpha=-\beta^{q^{n-m-l}-1}$ for some $\beta\in\mathbb F_{q^n}^*$, and
\[\ell_0(x)=-\beta^{q^{n-m-l}-1}x^{q^{n-m-l}}+x=-\beta^{-1}\big((\beta x)^{q^{n-m-l}}-\beta x\big)\]
with $\ker\ell_0=\beta^{-1}\mathbb F_{q^{m+l}}$. Likewise, $b=-a\gamma^{q^l-q^m}$ for some $\gamma\in\mathbb F_{q^n}^*$, so that
\[\ell(x)=ax^{q^m}-a\gamma^{q^l-q^m}x^{q^l}=a\gamma^{-q^m}\big((\gamma x)^{q^m}-(\gamma x)^{q^l}\big)\]
with $\im\ell=a\gamma^{-q^m}\ker\mathrm{Tr}_{m-l}$. It follows from \eqref{eq2} that
\[a^{q^l-q^m}=\frac{a^{q^l+1}}{a^{q^m+1}}=\frac{b^{q^m+1}}{a^{q^m+1}}=\gamma^{(q^l-q^m)(q^m+1)},\]
and thus $a=\gamma^{q^m+1}\delta$ for some $\delta\in\mathbb F_{q^{m-l}}^*$ with $b=-a\gamma^{q^l-q^m}=-\gamma^{q^l+1}\delta$. Furthermore,
\[-\beta^{1-q^{m+l}}=\alpha^{q^{m+l}}=b^{q^m}a^{-1}=-\gamma^{q^{m+l}-1}\delta^{q^m-1},\]
so
\begin{equation}\label{eq3}(\beta\gamma)^{1-q^{m+l}}=\delta^{q^m-1}.\end{equation}
% and
% \[\big(\delta^{q^m-1}\big)^\frac{q^n-1}{q^{m+l}-1}=1.\]
Recall that $\ker\ell_0\subseteq\im\ell$, which is
\[\beta^{-1}\mathbb F_{q^{m+l}}\subseteq a\gamma^{-q^m}\ker\mathrm{Tr}_{m-l}=\gamma\delta\ker\mathrm{Tr}_{m-l}=\gamma\ker\mathrm{Tr}_{m-l}.\]
In other words, $\mathrm{Tr}_{m-l}$ vanishes on $(\beta\gamma)^{-1}\mathbb F_{q^{m+l}}$. Clearly, $e=\lcm(m+l,m-l)$ divides $n$ and
\begin{equation}\label{eq11}\begin{split}&\mathrel{\phantom{\equiv}}\sum_{i=0}^{\frac e{m-l}-1}(\mathrm{Tr}_e((\beta\gamma)^{-1})x)^{q^{(m-l)i}}\\&\equiv\sum_{i=0}^{\frac e{m-l}-1}\mathrm{Tr}_e((\beta\gamma)^{-1}x)^{q^{(m-l)i}}\\&\equiv\mathrm{Tr}_{m-l}((\beta\gamma)^{-1}x)\\&\equiv0\pmod{x^{q^{m+l}}-x}.\end{split}\end{equation}
Here $x^{q^{(m-l)i}}\equiv x\pmod{x^{q^{m+l}}-x}$ if and only if $(m-l)i\equiv0\pmod{m+l}$; i.e., $i\equiv0\pmod{\frac e{m-l}}$. Then $\mathrm{Tr}_e((\beta\gamma)^{-1})=0$, which, combined with \eqref{eq3}, leads to
\[0=\sum_{i=0}^{\frac ne-1}(\beta\gamma)^{-q^{ei}}=(\beta\gamma)^{-1}\sum_{i=0}^{\frac ne-1}(\beta\gamma)^{1-q^{ei}}=(\beta\gamma)^{-1}\sum_{i=0}^{\frac ne-1}\big(\delta^{q^m-1}\big)^\frac{q^{ei}-1}{q^{m+l}-1}.\]
Let $\theta=\delta^{(q^m-1)\frac{q^e-1}{q^{m+l}-1}}$, which belongs to $\mathbb F_{q^e}^*$, so that
\begin{equation}\label{eq10}\sum_{i=0}^{\frac ne-1}\theta^i=\sum_{i=0}^{\frac ne-1}\theta^\frac{q^{ei}-1}{q^e-1}=0.\end{equation}
Then $\theta^\frac ne=1$; in addition, $\theta\ne1$ when $\gcd\big(\frac ne,q\big)=1$. If $v_2(l)\ge v_2(m-l)$, then $\theta=1$ by the claim, and necessarily $\gcd\big(\frac ne,q\big)\ne1$. Otherwise, $\theta=1$ if and only if $\delta^\frac{q^{m-l}-1}{q^d+1}=1$, and $\theta^\frac ne=1$ if and only if $\delta^{\frac{q^{m-l}-1}{q^d+1}\gcd(\frac ne,q^d+1)}=1$. This proves \ref{b_3}.

Assume the conditions in \ref{b_2}. Then
\[\ell(x)=ax^{q^m}+bx^{q^l}=\gamma\delta\big((\gamma x)^{q^m}-(\gamma x)^{q^l}\big),\]
with $\dim_{\mathbb F_q}\ker\ell=m-l$ and $\im\ell=\gamma\ker\mathrm{Tr}_{m-l}$. It is straightforward to verify that $a^{q^l+1}=b^{q^m+1}$, and $\alpha\ell(x)^{q^{n-m-l}}=\ell^\prime(x)$ with
\[\alpha^{q^{m+l}}=b^{q^m}a^{-1}=-\gamma^{q^{m+l}-1}\delta^{q^m-1}\]
from \eqref{eq2}. Now \eqref{eq6} applies and it remains to show that $\dim_{\mathbb F_q}\ker\ell_0=m+l$ with $\ker\ell_0\subseteq\im\ell$, where $\ell_0(x)=\alpha x^{q^{n-m-l}}+x$.
For $\theta=\delta^{(q^m-1)\frac{q^e-1}{q^{m+l}-1}}$, the condition on $\delta$ together with the claim implies $\sum_{i=0}^{\frac ne-1}\theta^i=0$, so
\[(-\alpha)^{q^{m+l}\frac{q^n-1}{q^{m+l}-1}}=\big(\gamma^{q^{m+l}-1}\delta^{q^m-1}\big)^\frac{q^n-1}{q^{m+l-1}}=\delta^{(q^m-1)\frac{q^e-1}{q^{m+l}-1}\frac{q^n-1}{q^e-1}}=\theta^\frac ne=1.\]
Then $\alpha=-\beta^{q^{n-m-l}-1}$ for some $\beta\in\mathbb F_{q^n}^*$, so that $(\beta\gamma)^{1-q^{m+l}}=\delta^{q^m-1}$ and $\ker\ell_0=\beta^{-1}\mathbb F_{q^{m+l}}$, with
\[\mathrm{Tr}_e((\beta\gamma)^{-1})=(\beta\gamma)^{-1}\sum_{i=0}^{\frac ne-1}\big(\delta^{q^m-1}\big)^\frac{q^{ei}-1}{q^{m+l}-1}=(\beta\gamma)^{-1}\sum_{i=0}^{\frac ne-1}\theta^i=0.\]
By the same argument as in \eqref{eq11}, we have $\ker\ell_0\subseteq\im\ell$ and establish \ref{b_2}.
\end{proof}

Recall that
\[\sum_{w\in\mathbb F_{q^n}}\chi(\mathrm{Tr}(w\ell(w)))=\epsilon\mathcal G^rq^{n-r},\]
where $r$ is the rank of the quadratic form, $\chi$ is the canonical additive character of $\mathbb F_q$ and $\mathcal G$ is the quadratic Gauss sum of $\mathbb F_q$. The quadratic form from one of the polynomials in the above theorem can be further investigated.

\begin{proposition}
Let $\ell(x)=\gamma^{q^m+1}\delta x^{q^m}-\gamma^{q^l+1}\delta x^{q^l}$ for $\gamma\in\mathbb F_{q^n}^*$ and $\delta\in\mathbb F_{q^{m-l}}^*$ with $0<l<m<\frac n2$, $v_2(l)\ge v_2(m-l)$, $n$ divisible by both $m+l$ and $m-l$, and $\gcd\big(\frac n{\lcm(m+l,m-l)},q\big)\ne1$. Then $\epsilon=(-1)^{n-1+\frac{q-1}2(\frac nd-1)m}\eta(\mathrm N(2\delta))$ for $d=\gcd(m,m-l)$.
\end{proposition}
\begin{proof}
We show that there exists some $c\in\mathbb F_{q^{m-l}}^*$ with $\delta c^{q^m+1}\in\mathbb F_{q^d}$, so that
\begin{equation}\label{eq12}\begin{split}x\ell(x)&=\gamma^{q^m+1}\delta c^{q^m+1}(c^{-1}x)^{q^m+1}-\gamma^{q^l+1}\delta c^{q^l+1}(c^{-1}x)^{q^l+1}\\&=\delta c^{q^m+1}\big((\gamma c^{-1}x)^{q^m+1}-(\gamma c^{-1}x)^{q^l+1}\big).\end{split}\end{equation}
Note that $\gcd(q^m+1,q^{m-l}-1)$ divides
\[\gcd(q^{2m}-1,q^{m-l}-1)=q^{\gcd(2m,m-l)}-1=q^d-1,\]
so $\gcd(q^m+1,q^{m-l}-1)$ divides $\gcd(q^m+1,q^d-1)=2$ as $q^m+1\equiv2\pmod{q^d-1}$. Then
\[\gcd(q^m+1,q^{m-l}-1)=2,\]
and it suffices to find an element $c_0\in\mathbb F_{q^{m-l}}^*$ such that $\delta c_0^2\in\mathbb F_{q^d}$. Take $\alpha\in\mathbb F_{q^d}^*$ that is a non-square element in $\mathbb F_{q^{m-l}}$, so that one may let either $c_0^{-2}=\delta$ or $c_0^{-2}=\delta\alpha$ to get $\delta c_0^2$ lying in $\mathbb F_{q^d}$. Such $\alpha$ exists in $\mathbb F_{q^d}^*$ since $\alpha^\frac{q^{m-l}-1}2=\alpha^{\frac{q^d-1}2\frac{m-l}d}$, while $v_2(d)=\min\{v_2(l),v_2(m-l)\}=v_2(m-l)$ and $\frac{m-l}d$ is odd.

For the canonical additive character $\chi_d$ of $\mathbb F_{q^d}$, it follows from \eqref{eq12} that
\[\epsilon\mathcal G^{n-2m}q^{2m}=\sum_{w\in\mathbb F_{q^n}}\chi(\mathrm{Tr}(w\ell(w)))=\sum_{w\in\mathbb F_{q^n}}\chi_d\big(\mathrm{Tr}_d\big(\delta c^{q^m+1}w\big(w^{q^m}-w^{q^l}\big)\big)\big).\]
Now $\mathrm{Tr}_d\big(\delta c^{q^m+1}x\big(x^{q^m}-x^{q^l}\big)\big)$ is a quadratic form of $\mathbb F_{q^n}/\mathbb F_{q^d}$ with coefficients in $\mathbb F_{q^d}$. Taking the absolute value into account, we have
\[\sum_{w\in\mathbb F_{q^n}}\chi_d\big(\mathrm{Tr}_d\big(\delta c^{q^m+1}w\big(w^{q^m}-w^{q^l}\big)\big)\big)=\epsilon_d\mathcal G_d^\frac{n-2m}dq^{2m},\]
where $\mathcal G_d=(-1)^{d-1}\mathcal G^d$ (\cite[Theorem 5.15]{lidl1997}) and
\[\epsilon_d=(-1)^{(\frac nd-1)\big(\frac{q^d-1}2\frac md+1\big)}\eta\big(\mathrm N\big(2\delta c^{q^m+1}\big)\big)=(-1)^{(\frac nd-1)(\frac{q-1}2m+1)}\eta(\mathrm N(2\delta)),\]
as a consequence of \ref{e_4} in Theorem \ref{e}. Altogether, it gives
\[\begin{split}&\mathrel{\phantom{=}}\epsilon\mathcal G^{n-2m}q^{2m}\\&=(-1)^{(\frac nd-1)(\frac{q-1}2m+1)}\eta(\mathrm N(2\delta))((-1)^{d-1}\mathcal G^d)^\frac{n-2m}dq^{2m}\\&=(-1)^{n-1+\frac{q-1}2(\frac nd-1)m}\eta(\mathrm N(2\delta))\mathcal G^{n-2m}q^{2m},\end{split}\]
which completes the proof.
\end{proof}
% $\delta c^{q^m+1}\in\mathbb F_{q^d}$, $\delta^{q^d-1}c^{(q^m+1)(q^d-1)}=1$
% $\delta^{\frac{q^{m-l}-1}{q^d+1}\frac n{m_0}}=1$,
% $\gcd(\frac{q^{m-l}-1}{q^d+1}\gcd(q^d+1,\frac n{m_0}),q^d-1)=q^d-1$,
% the order of $\delta^{q^d-1}$: $\frac{q^{m-l}-1}{q^{d_0}-1}\gcd(q^d+1,\frac n{m_0})$
% the order of $c$: $\gcd((q^{m-l}-1)(q^d-1),q^n-1)=(q^{m-l}-1)\gcd((q^d-1),\frac{q^n-1}{q^{m-l}-1})=(q^{m-l}-1)\gcd(q^d-1,\frac n{m-l})$
% the order of $c^{(q^m+1)(q^d-1)}$: $\gcd((q^m+1)(q^d-1),(q^{m-l}-1)\gcd(q^d-1,\frac n{m-l}))=\gcd((q^m+1)(q^d-1)/\gcd,q^{m-l}-1)\gcd(q^d-1,\frac n{m-l})$, $\frac{q^{m-l}-1}{\gcd((q^m+1)(q^d-1)/\gcd,q^{m-l}-1)}$.
% $\frac{q^{d_0}-1}{\gcd(q^d+1,\frac n{m_0})}$, $\gcd((q^m+1)(q^d-1)/\gcd,q^{m-l}-1)=(q^d+1)(q^d-1)/\gcd=(q^{d_0}-1)/\gcd$

So far, we have been studying Artin-Schreier curves of the form $y^q-y-x\ell(x)$ from quadratic forms of $\mathbb F_{q^n}/\mathbb F_q$. One may be interested in curves defined by $y^{q^k}-y-x\ell(x)$ for some positive divisor $k$ of $n$, even if $\mathrm{Tr}_k(x\ell(x))$ is not a quadratic form of $\mathbb F_{q^n}/\mathbb F_{q^k}$. Here is an example.

\begin{corollary}
Continue with the notation and assumptions in the last proposition. For a positive divisor $k$ of $m-l$, the number of rational points of the curve over $\mathbb F_{q^n}$ defined by $y^{q^k}-y-x\ell(x)$ is
\[\begin{cases}q^n+1-(-1)^\frac{(q-1)n}4\eta(\mathrm N(\delta))(q^k-1)q^{\frac n2+m}&\text{if }\frac nk\text{ is even,}\\q^n+1&\text{if }\frac nk\text{ is odd.}\end{cases}\]
It is a maximal or minimal curve if $\frac nk$ is even.
\end{corollary}
\begin{proof}
As is easily checked, the curve has genus $\frac{(q^k-1)q^m}2$, with exactly one point at infinity. The number of its other rational points over $\mathbb F_{q^n}$ is
\[\begin{split}&\mathrel{\phantom{=}}\sum_{w\in\mathbb F_{q^n}}\sum_{u\in\mathbb F_{q^k}}\chi_k(u\mathrm{Tr}_k(w\ell(w)))\\&=\sum_{u\in\mathbb F_{q^k}}\sum_{w\in\mathbb F_{q^n}}\chi(\mathrm{Tr}(uw\ell(w)))\\&=q^n+\sum_{u\in\mathbb F_{q^k}^*}(-1)^{n-1+\frac{q-1}2(\frac nd-1)m}\eta(\mathrm N(2u\delta))\mathcal G^{n-2m}q^{2m},\end{split}\]
where $\mathcal G^{n-2m}=(-1)^\frac{(q-1)(n-2m)}4q^{\frac n2-m}$ when $n$ is even, and
\[\sum_{u\in\mathbb F_{q^k}^*}\eta(\mathrm N(2u))=\begin{cases}q^k-1&\text{if }\frac nk\text{ is even,}\\0&\text{if }\frac nk\text{ is odd.}\end{cases}\]
A simple calculation yields the result.
\end{proof}

As a final remark, we introduce an alternative approach towards maximal and minimal Artin-Schreier curves.

\begin{proposition}
Let $L$ be a $q$-linear polynomial over $\mathbb F_{q^n}$ of degree less than $q^n$. Then $L=L^\prime$ with $\dim_{\mathbb F_q}\im L=r$ for some positive integer $r$ if and only if
\[L(x)=\sum_{i=1}^rd_i\alpha_i\mathrm{Tr}(\alpha_ix)\]
for some $d_i\in\mathbb F_q^*$ and $\alpha_1,\dots,\alpha_r$ that are linearly independent in $\mathbb F_{q^n}/\mathbb F_q$. In this case, $\epsilon=\eta(d_1\cdots d_r)$ for $L=\frac{\ell+\ell^\prime}2$.
\end{proposition}
\begin{proof}
The sufficiency is obvious. If $L$ has the above form, then
\[\mathrm{Tr}(x\ell(x))=\mathrm{Tr}(xL(x))\equiv\sum_{i=1}^rd_i\mathrm{Tr}(\alpha_ix)^2\pmod{x^{q^n}-x},\]
where $(\mathrm{Tr}(\alpha_1x),\dots,\mathrm{Tr}(\alpha_rx))$ maps $\mathbb F_{q^n}$ onto $\mathbb F_q^r$, and thus $\epsilon=\eta(d_1\cdots d_r)$ by definition.

Conversely, suppose that $L=L^\prime$ and $\dim_{\mathbb F_q}\im L=r$. Let $\beta_1,\dots,\beta_r$ be a basis of $\im L$ over $\mathbb F_q$. Then there exist $\alpha_1,\dots,\alpha_r\in\mathbb F_{q^n}$ such that
\[L(x)=\beta_1\mathrm{Tr}(\alpha_1x)+\dots+\beta_r\mathrm{Tr}(\alpha_rx),\]
for every linear form of $\mathbb F_{q^n}/\mathbb F_q$ is $\mathrm{Tr}(\alpha x)$ for some $\alpha\in\mathbb F_{q^n}$. Let $W$ be the subspace of $\mathbb F_{q^n}/\mathbb F_q$ spanned by $\alpha_1,\dots,\alpha_r$. Clearly, $\ker L=W^\perp$ and then $\dim_{\mathbb F_q}W=r$, so $\alpha_1,\dots,\alpha_r$ is a basis of $(\ker L)^\perp$ over $\mathbb F_q$, where $(\ker L)^\perp=(\ker L^\prime)^\perp=\im L$. Hence, there exists a nonsingular matrix $C$ over $\mathbb F_q$ such that
\[\begin{pmatrix}\beta_1&\cdots&\beta_r\end{pmatrix}=\begin{pmatrix}\alpha_1&\cdots&\alpha_r\end{pmatrix}C.\]
Then
\[L(x)=\begin{pmatrix}\alpha_1&\cdots&\alpha_r\end{pmatrix}C\begin{pmatrix}\mathrm{Tr}(\alpha_1x)\\\vdots\\\mathrm{Tr}(\alpha_rx)\end{pmatrix}.\]
Observe that
\[L^\prime(x)=\begin{pmatrix}\alpha_1&\cdots&\alpha_r\end{pmatrix}C^\mathrm T\begin{pmatrix}\mathrm{Tr}(\alpha_1x)\\\vdots\\\mathrm{Tr}(\alpha_rx)\end{pmatrix},\]
and there exist $x_1,\dots,x_r\in\mathbb F_{q^n}$ such that
\[\mathrm{Tr}(\alpha_ix_j)=\begin{cases}1&\text{if }i=j,\\0&\text{if }i\ne j.\end{cases}\]
Since $L(x_j)=L^\prime(x_j)$ for $j=1,\dots,n$, one has
\[\begin{pmatrix}\alpha_1&\cdots&\alpha_r\end{pmatrix}C=\begin{pmatrix}\alpha_1&\cdots&\alpha_r\end{pmatrix}C^\mathrm T.\]
This implies $C=C^\mathrm T$, and $C=P^\mathrm TDP$ for a nonsingular matrix $P$ over $\mathbb F_q$ and a diagonal matrix with nonzero entries $d_1,\dots,d_r\in\mathbb F_q$. As a result,
\[L(x)=\begin{pmatrix}\alpha_1&\cdots&\alpha_r\end{pmatrix}P^\mathrm TDP\begin{pmatrix}\mathrm{Tr}(\alpha_1x)\\\vdots\\\mathrm{Tr}(\alpha_rx)\end{pmatrix}.\]
Substituting $\begin{pmatrix}\alpha_1&\cdots&\alpha_r\end{pmatrix}$ for $\begin{pmatrix}\alpha_1&\cdots&\alpha_r\end{pmatrix}P^\mathrm T$, we get the desired basis of $\im L$.
\end{proof}

As we have seen, if $L=\frac{\ell+\ell^\prime}2$ for a $q$-linear polynomial $\ell$ over $\mathbb F_{q^n}$ of degree less than $q^n$, then $L=L^\prime$, and vice versa. Now, assume that $n$ is even. Except the trivial cases mentioned at the beginning of this section, every maximal or minimal curve over $\mathbb F_{q^n}$ defined by $y^q-y-x\ell(x)$ arises from
\[L(x)=\sum_{i=1}^rd_i\alpha_i\mathrm{Tr}(\alpha_ix)=\sum_{i=0}^{n-1}a_ix^{q^i},\]
where $r$ is even with $d_i\in\mathbb F_q^*$ and $\alpha_1,\dots,\alpha_r$ linearly independent in $\mathbb F_{q^n}/\mathbb F_q$, such that $a_{\frac{n-r}2+1},a_{\frac{n-r}2+2},\dots,a_\frac n2$ are all zero. Then we may take the part of degree $q^{\frac{n-r}2}$ in $L$. Whether it is maximal or minimal depends on $(-1)^\frac{(q-1)r}4\eta(d_1\cdots d_r)$. For example, let
\[L(x)=d_1\alpha_1\mathrm{Tr}(\alpha_1x)+d_2\alpha_2\mathrm{Tr}(\alpha_2x)=\sum_{i=0}^{n-1}\big(d_1\alpha_1^{q^i+1}+d_2\alpha_2^{q^i+1}\big)x^{q^i},\]
with $d_1,d_2\in\mathbb F_q^*$, $\alpha_1\alpha_2^{-1}\notin\mathbb F_q$ and $d_1\alpha_1^{q^\frac n2+1}+d_2\alpha_2^{q^\frac n2+1}=0$, and let
\[\ell(x)=2\sum_{i=0}^{\frac n2-1}\big(d_1\alpha_1^{q^i+1}+d_2\alpha_2^{q^i+1}\big)x^{q^i},\]
so that $L=\frac{\ell+\ell^\prime}2$. Here $\deg\ell=q^{\frac n2-1}$ since $L(x)^{q^{\frac n2-1}}$ is congruent to a polynomial of degree at most $q^{n-2}$ modulo $x^{q^n}-x$, while $\dim_{\mathbb F_q}\ker L=n-2$. That describes all the maximal and minimal Artin-Schreier curves from quadratic forms of $\mathbb F_{q^n}/\mathbb F_q$ of rank $2$ (cf. \cite[Theorem 5.5]{anbar2014quadratic} and \cite{bartoli2021explicit}).

% $q$ even?
% Calculating $\prod_{j=1}^rG(\omega_j)$ by resultant

\section*{Acknowledgement}

The author would like to express his sincere gratitude to Prof. Sihem Mesnager for her advice.

\bibliographystyle{abbrv}
\bibliography{references}

\end{document}